\documentclass[11pt,a4paper]{amsart}
\usepackage{mathrsfs}
\usepackage{syntonly}
\usepackage{amsmath}
\usepackage{amsthm}
\usepackage{amsfonts}
\usepackage{amssymb}
\usepackage{latexsym}
\usepackage{amscd,amssymb,amsopn,amsmath,amsthm,graphics,amsfonts,mathrsfs,accents,enumerate,verbatim,calc}
\usepackage[dvips]{graphicx}
\usepackage[colorlinks=true,linkcolor=red,citecolor=blue]{hyperref}
\usepackage[all]{xy}
\usepackage{enumitem}

\date{}
\allowdisplaybreaks[4] \footskip=15pt
\renewcommand{\uppercasenonmath}[1]{}

\numberwithin{equation}{section} \theoremstyle{plain}
\newtheorem{lem}{Lemma}[section]
\newtheorem{cor}[lem]{Corollary}
\newtheorem{prop}[lem]{Proposition}
\newtheorem{thm}[lem]{Theorem}

\newtheorem{cond}[lem]{Condition}
\newtheorem{definition}[lem]{Definition}
\newtheorem{Ex}[lem]{Example}
\newtheorem{Quest}[lem]{Question}
\newtheorem{Property}[lem]{Property}
\newtheorem{Properties}[lem]{Properties}
\newtheorem{Subprops}{}[lem]
\newtheorem{Para}[lem]{}

\newtheorem{fact}[lem]{Fact}

\newtheorem{remark}[lem]{Remark}
\newtheorem{rem}[lem]{Remark}

\newtheorem*{ack*}{ACKNOWLEDGEMENTS}

\newcommand{\pf}{\noindent\begin {proof}}
\newcommand{\epf}{\end{proof}}

\newcommand{\C}{\mathcal{C}}

\begin{document}
\begin{center}
{\Large  \bf  Proper resolutions and Gorensteinness in extriangulated categories}

\vspace{0.5cm}  Jiangsheng Hu$^{a}$, Dongdong Zhang$^{b}$\footnote{Corresponding author. Jiangsheng Hu was supported by the NSF of China (Grant No. 11501257, 11671069, 11771212),  Qing Lan Project of Jiangsu Province and Jiangsu Government Scholarship for Overseas Studies (JS-2019-328).  Panyue Zhou was supported by the NSF of China (Grant No. 11901190 and 11671221), the Hunan Provincial Natural Science Foundation of China (Grant No. 2018JJ3205) and the Scientific Research Fund of Hunan Provincial Education Department (Grant No. 19B239).} and Panyue Zhou$^{c}$ \\
\medskip

\hspace{-4mm}$^{a}$School of Mathematics and Physics, Jiangsu University of Technology,
 Changzhou 213001, China\\
 $^b$Department of Mathematics, Zhejiang Normal University,
\small Jinhua 321004, China\\
$^c$College of Mathematics, Hunan Institute of Science and Technology, Yueyang, Hunan 414006, China\\
E-mails: jiangshenghu@jsut.edu.cn, zdd@zjnu.cn and panyuezhou@163.com \\
\end{center}

\bigskip
\medskip
\centerline { \bf  Abstract}
\leftskip10truemm \rightskip10truemm \noindent
\hspace{1em} Let $(\mathcal{C},\mathbb{E},\mathfrak{s})$ be an extriangulated category with a proper class $\xi$ of $\mathbb{E}$-triangles, and $\mathcal{W}$
an additive full subcategory of $\C$. We provide a method
for constructing a proper $\mathcal{W}(\xi)$-resolution (respectively, coproper
$\mathcal{W}(\xi)$-coresolution) of one term in an $\mathbb{E}$-triangle in $\xi$ from that of the other two
terms.  By using this way, we establish the stability of the Gorenstein
category $\mathcal{GW}(\xi)$ in extriangulated categories. These results generalise their work by Huang and Yang-Wang, but the proof is not too
far from their case. Finally, we give some applications about our main results.
\\[2mm]
{\bf Keywords:}
  extriangulated categories; proper resolution; coproper coresolution; Gorenstein categories.\\
{\bf 2010 Mathematics Subject Classification:} 18E30; 18E10; 18G25;  18G10.

\leftskip0truemm \rightskip0truemm
\section { \bf Introduction}

Let $\mathcal{A}$ be an abelian category and $\mathcal {W}$ an additive full subcategory
of $\mathcal{A}$. Huang \cite{Hua} provided a method for constructing a proper $\mathcal{W}$-resolution (respectively, coproper $\mathcal{W}$-coresolution) of one
term in a short exact sequence in $\mathcal{A}$ from those of the other two terms.
By using these, he affirmatively answered an open question on the stability of the Gorenstein category $\mathcal{G}(\mathcal{W})$ posed by Sather-Wagstaff, Sharif
and White \cite{Sather} and also proved that $\mathcal{G}(\mathcal{W})$ is closed under direct summands.
Later, Yang-Wang \cite{Yang3} extended Huang's results to triangulated categories in parallel.
Some
further investigations of proper resolutions (coproper coresoltutions) and Gorenstein categories for abelian categories or triangulated
categories can be seen in \cite{Hua,Maxin,RL1,Wang,Yang1,Yang2}.

The notion of extriangulated categories was introduced by Nakaoka and Palu in \cite{NP} as a simultaneous generalization of
exact categories and triangulated categories. Exact categories and extension closed subcategories of an
extriangulated category are extriangulated categories, while there exist some other examples of extriangulated categories which are neither exact nor triangulated, see \cite{NP,ZZ,HZZ}. Hence many results hold on exact categories
and triangulated categories can be unified in the same framework.
Based on this idea, we will unify the results of Huang and Yang-Wang in the framework of extriangulated categories.

Let $(\mathcal{C},\mathbb{E},\mathfrak{s})$  be an extriangulated category with a proper class $\xi$ of $\mathbb{E}$-triangles. The authors \cite{HZZ} studied a relative homological algebra in $\mathcal C$ which parallels the relative homological algebra in a triangulated category. By specifying a class of $\mathbb{E}$-triangles, which is called a proper class $\xi$ of
$\mathbb{E}$-triangles, we introduced $\xi$-$\mathcal{G}$projective dimensions and  $\xi$-$\mathcal{G}$injective dimensions,
and discussed their properties.
Inspired by Huang and Yang-Wang's work, in this paper we introduce and study Gorenstein category in extriangulated categories and demonstrate that this category shares some basic properties with Gorenstein category in the abelian category or in the triangulated category.

This paper is organized as follows: Section 2 gives some preliminaries and basic facts about extriangulated categories which will be used throughout the paper. Section 3 provides a method
for constructing a proper resolution (respectively, coproper
coresolution) of one term in an $\mathbb{E}$-triangle in $\xi$ from those of the other two
terms, which generalizes Huang's results on an abelian category and Yang-Wang's results on triangulated category and is new for an exact category case (see Theorems \ref{thm1}, \ref{thm2}, \ref{thm3}, \ref{thm4} and Remark \ref{rem:3.11}).  Section 4 is devoted to studying the Gorenstein category of extriangulated categories. More precisely, we prove that this Gorenstein category is closed under direct summands and
the stability of the Gorenstein category is also established in in extriangulated categories, which
refines a result of Yang and Wang (see Theorem \ref{thm:stability} and Remark \ref{rem:4.19}).

\section{\bf Preliminaries}
Throughout this paper, we always assume that  $\mathcal{C}=(\mathcal{C}, \mathbb{E}, \mathfrak{s})$ is an extriangulated category with enough $\xi$-projectives and enough $\xi$-injectives, and it satisfies Condition (WIC) (for details, see Condition \ref{cond:4.11}).  We also assume that $\xi$ is a proper class of $\mathbb{E}$-triangles in  $(\mathcal{C}, \mathbb{E}, \mathfrak{s})$.

Let us briefly recall some definitions and basic properties of extriangulated categories from \cite{NP}.
We omit some details here, but the reader can find
them in \cite{NP}.

Let $\mathcal{C}$ be an additive category equipped with an additive bifunctor
$$\mathbb{E}: \mathcal{C}^{\rm op}\times \mathcal{C}\rightarrow {\rm Ab},$$
where ${\rm Ab}$ is the category of abelian groups. For any objects $A, C\in\mathcal{C}$, an element $\delta\in \mathbb{E}(C,A)$ is called an $\mathbb{E}$-extension.
Let $\mathfrak{s}$ be a correspondence which associates an equivalence class $$\mathfrak{s}(\delta)=\xymatrix@C=0.8cm{[A\ar[r]^x
 &B\ar[r]^y&C]}$$ to any $\mathbb{E}$-extension $\delta\in\mathbb{E}(C, A)$. This $\mathfrak{s}$ is called a {\it realization} of $\mathbb{E}$, if it makes the diagrams in \cite[Definition 2.9]{NP} commutative.
 A triplet $(\mathcal{C}, \mathbb{E}, \mathfrak{s})$ is called an {\it extriangulated category} if it satisfies the following conditions.
\begin{enumerate}
\item $\mathbb{E}\colon\mathcal{C}^{\rm op}\times \mathcal{C}\rightarrow \rm{Ab}$ is an additive bifunctor.

\item $\mathfrak{s}$ is an additive realization of $\mathbb{E}$.

\item $\mathbb{E}$ and $\mathfrak{s}$  satisfy the compatibility conditions in \cite[Definition 2.12]{NP}.

 \end{enumerate}

\begin{rem}
Note that both exact categories and triangulated categories are extriangulated categories, see \cite[Example 2.13]{NP} and extension closed subcategories of extriangulated categories are
again extriangulated, see \cite[Remark 2.18]{NP}. Moreover, there exist extriangulated categories which
are neither exact categories nor triangulated categories, see \cite[Proposition 3.30]{NP}, \cite[Example 4.14]{ZZ} and \cite[Remark 3.3]{HZZ}.
\end{rem}

\begin{lem}\label{lem1} \emph{(see \cite[Proposition 3.15]{NP})} Assume that $(\mathcal{C}, \mathbb{E},\mathfrak{s})$ is an extriangulated category. Let $C$ be any object, and let $\xymatrix@C=2em{A_1\ar[r]^{x_1}&B_1\ar[r]^{y_1}&C\ar@{-->}[r]^{\delta_1}&}$ and $\xymatrix@C=2em{A_2\ar[r]^{x_2}&B_2\ar[r]^{y_2}&C\ar@{-->}[r]^{\delta_2}&}$ be any pair of $\mathbb{E}$-triangles. Then there is a commutative diagram
in $\mathcal{C}$
$$\xymatrix{
    & A_2\ar[d]_{m_2} \ar@{=}[r] & A_2 \ar[d]^{x_2} \\
  A_1 \ar@{=}[d] \ar[r]^{m_1} & M \ar[d]_{e_2} \ar[r]^{e_1} & B_2\ar[d]^{y_2} \\
  A_1 \ar[r]^{x_1} & B_1\ar[r]^{y_1} & C   }
  $$
  which satisfies $\mathfrak{s}(y^*_2\delta_1)=\xymatrix@C=2em{[A_1\ar[r]^{m_1}&M\ar[r]^{e_1}&B_2]}$ and
  $\mathfrak{s}(y^*_1\delta_2)=\xymatrix@C=2em{[A_2\ar[r]^{m_2}&M\ar[r]^{e_2}&B_1]}$.

\end{lem}

A class of $\mathbb{E}$-triangles $\xi$ is called {\it saturated} if in the situation of Lemma \ref{lem1}, whenever { the $\mathbb{E}$-triangles $\xymatrix@C=2em{A_2\ar[r]^{x_2}&B_2\ar[r]^{y_2}&C\ar@{-->}[r]^{\delta_2 }&}$ and $\xymatrix@C=2em{A_1\ar[r]^{m_1}&M\ar[r]^{e_1}&B_2\ar@{-->}[r]^{y_2^*\delta_1 }&}$} belong to $\xi$, then the  $\mathbb{E}$-triangle $\xymatrix@C=2em{A_1\ar[r]^{x_1}&B_1\ar[r]^{y_1}&C\ar@{-->}[r]^{\delta_1 }&}$  belongs to $\xi$.

An $\mathbb{E}$-triangle $\xymatrix@C=2em{A\ar[r]^x&B\ar[r]^y&C\ar@{-->}[r]^{\delta}&}$ is called {\it split} if $\delta=0$. By \cite[Corollary 3.5]{NP}, we know that it is split if and only if $x$ is section if and only if  $y$ is retraction. The full subcategory  consisting of the split $\mathbb{E}$-triangles will be denoted by $\Delta_0$.

A class of $\mathbb{E}$-triangles $\xi$ is {\it closed under base change} if for any $\mathbb{E}$-triangle $$\xymatrix@C=2em{A\ar[r]^x&B\ar[r]^y&C\ar@{-->}[r]^{\delta}&\in\xi}$$ and any morphism $c\colon C' \to C$, then any $\mathbb{E}$-triangle  $\xymatrix@C=2em{A\ar[r]^{x'}&B'\ar[r]^{y'}&C'\ar@{-->}[r]^{c^*\delta}&}$ belongs to $\xi$.

Dually, a class of  $\mathbb{E}$-triangles $\xi$ is {\it closed under cobase change} if for any $\mathbb{E}$-triangle $$\xymatrix@C=2em{A\ar[r]^x&B\ar[r]^y&C\ar@{-->}[r]^{\delta}&\in\xi}$$ and any morphism $a\colon A \to A'$, then any $\mathbb{E}$-triangle  $\xymatrix@C=2em{A'\ar[r]^{x'}&B'\ar[r]^{y'}&C\ar@{-->}[r]^{a_*\delta}&}$ belongs to $\xi$.

  \begin{definition} \cite[Definition 3.1]{HZZ}\label{def:proper class} {\rm Let $\xi$ be a class of $\mathbb{E}$-triangles which is closed under isomorphisms. $\xi$ is called a {\it proper class} of $\mathbb{E}$-triangles if the following conditions hold:

  (1) $\xi$ is closed under finite coproducts and $\Delta_0\subseteq \xi$.

  (2) $\xi$ is closed under base change and cobase change.

  (3) $\xi$ is saturated.}
  \end{definition}
 \begin{definition} \cite[Definition 4.1]{HZZ}
 {\rm An object $P\in\mathcal{C}$  is called {\it $\xi$-projective}  if for any $\mathbb{E}$-triangle $$\xymatrix{A\ar[r]^x& B\ar[r]^y& C \ar@{-->}[r]^{\delta}& }$$ in $\xi$, the induced sequence of abelian groups $\xymatrix@C=0.6cm{0\ar[r]& \mathcal{C}(P,A)\ar[r]& \mathcal{C}(P,B)\ar[r]&\mathcal{C}(P,C)\ar[r]& 0}$ is exact. Dually, we have the definition of {\it $\xi$-injective}.}
\end{definition}

We denote $\mathcal{P(\xi)}$ (respectively $\mathcal{I(\xi)}$) the class of $\xi$-projective (respectively $\xi$-injective) objects of $\mathcal{C}$. It follows from the definition that this subcategory $\mathcal{P}(\xi)$ and $\mathcal{I}(\xi)$ are full, additive, closed under isomorphisms and direct summands.

 An extriangulated  category $(\mathcal{C}, \mathbb{E}, \mathfrak{s})$ is said to  have {\it  enough
$\xi$-projectives} \ (respectively {\it  enough $\xi$-injectives}) provided that for each object $A$ there exists an $\mathbb{E}$-triangle
 $\xymatrix@C=2.1em{K\ar[r]& P\ar[r]&A\ar@{-->}[r]& }$ (respectively $\xymatrix@C=2em{A\ar[r]& I\ar[r]& K\ar@{-->}[r]&}$) in $\xi$ with $P\in\mathcal{P}(\xi)$
 (respectively $I\in\mathcal{I}(\xi)$).

 {\rm Let $\xymatrix@C=2.1em{A\ar[r]^x& B\ar[r]^y&C\ar@{-->}[r]^{\delta}& }$ be an $\mathbb{E}$-triangle (in $\xi$).
 The morphism $x: A\rightarrow B$ is called ($\xi$-){\it inflation}, and $y: B\rightarrow C$ is called   ($\xi$-){\it deflation};
  $x$ is called the {\it hokernel} of $y$ and $y$ is called the {\it hocokernel} of $x$.}
Let $\mathcal{W}$ be a class of objects in $\mathcal{C}$. We say that $\mathcal{W}$ is closed {\it under hokernels of $\xi$-deflation} if,
 whenever the $\mathbb{E}$-triangle $\xymatrix@C=2.1em{A\ar[r]^x& B\ar[r]^y&C\ar@{-->}[r]^{\delta}& }$ in $\xi$ with $B, C\in\mathcal{W}$,
 then $A\in\mathcal{W}$. Dually, we say that $\mathcal{W}$ is closed {\it under hocokernels of $\xi$-inflation} if,
 whenever the $\mathbb{E}$-triangle $\xymatrix@C=2.1em{A\ar[r]^x& B\ar[r]^y&C\ar@{-->}[r]^{\delta}& }$ in $\xi$ with $A, B\in\mathcal{W}$,
 then $C\in\mathcal{W}$.

In addition, we assume the following condition for { the rest of the paper} (see \cite[Condition 5.8]{NP}).

\begin{cond} \label{cond:4.11} \emph{({\rm Condition (WIC)})}  Consider the following conditions.

\emph{(1)} Let $f\in\mathcal{C}(A, B), g\in\mathcal{C}(B, C)$ be any composable pair of morphisms. If $gf$ is an inflation, then so is $f$.

\emph{(2)} Let $f\in\mathcal{C}(A, B), g\in\mathcal{C}(B, C)$ be any composable pair of morphisms. If $gf$ is a deflation, then so is $g$.
\end{cond}

 \begin{fact}
 { \emph{(1)} The class of $\xi$-inflations \emph{(}respectively $\xi$-deflations{\rm )} is closed under compositions} \emph{(}see \cite[Corollary 3.5]{HZZ}\emph{)}.

 \emph{(2)} { Assume that} $x: A\rightarrow B$ and $y: B\rightarrow C$ are composable pair of morphisms,
 then $x$ is $\xi$-inflation {\rm(}respectively { $y$ is a $\xi$-deflation}{\rm)} whenever $yx$ is a $\xi$-inflation {\rm (}respectively  $\xi$-deflation{\rm )} \emph{(}see \cite[Proposition 4.13]{HZZ}\emph{)}.
 \end{fact}


 Similar to the proof of \cite[Lemma 4.15]{HZZ}, we have the following result.
 \begin{prop}\label{prop}  { Let} $\xymatrix@C=2em{A\ar[r]^x& B\ar[r]^y& C\ar@{-->}[r]^\delta&}$ and
 $\xymatrix@C=2em{A'\ar[r]^{x'}& B'\ar[r]^{y'}& C'\ar@{-->}[r]^{\delta'}&}$ be $\mathbb{E}$-triangles in $\xi$.
 If $(a, c): \delta\rightarrow \delta'$ is a morphism of $\mathbb{E}$-triangles where $a, c$ are $\xi$-inflations
 {\rm(}respectively $\xi$-deflations{\rm)}, then there is a $\xi$-inflation {\rm(}respectively $\xi$-deflation{\rm)}
$b: B\rightarrow B'$ which { makes} the following diagram commutative
$$\xymatrix{A\ar[r]^x\ar[d]_a&B\ar[r]^y\ar[d]^b&C\ar[d]^c\ar@{-->}[r]^{\delta}&\\
A'\ar[r]^{x'}& B'\ar[r]^{y'}& C'\ar@{-->}[r]^{\delta'}&.}$$
\end{prop}

\begin{definition} \cite[Definition 4.4]{HZZ}
{\rm A  complex $$\xymatrix@C=2em{\cdots\ar[r]&X_1\ar[r]^{d_1}&X_0\ar[r]^{d_0}&X_{-1}\ar[r]&\cdots}$$ in $\mathcal{C}$ is called {\it $\xi$-exact complex}
if  for each integer $n$, there exists an $\mathbb{E}$-triangle $$\xymatrix{K_{n+1}\ar[r]^{g_n}&X_n\ar[r]^{f_n}&K_n\ar@{-->}[r]^{\delta_n}&}$$
in $\xi$ and $d_n=g_{n-1}f_n$.}\end{definition}


\begin{definition} \cite[Definition 4.5]{HZZ}
{\rm Let $\mathcal{W}$ be a class of objects in $\mathcal{C}$. An $\mathbb{E}$-triangle
$$\xymatrix@C=2em{A\ar[r]& B\ar[r]& C\ar@{-->}[r]& }$$ in $\xi$ is called to be
{\it $\mathcal{C}(-,\mathcal{W})$-exact} (respectively
{\it $\mathcal{C}(\mathcal{W},-)$-exact}) if for any $W\in\mathcal{W}$, the induced sequence of abelian group
$\xymatrix@C=2em{0\ar[r]&\mathcal{C}(C,W)\ar[r]&\mathcal{C}(B,W)\ar[r]&\mathcal{C}(A,W)\ar[r]& 0}$
(respectively \\ $\xymatrix@C=2em{0\ar[r]&\mathcal{C}(W,A)\ar[r]&\mathcal{C}(W,B)\ar[r]&\mathcal{C}(W,C)\ar[r]&0}$) is exact in ${\rm Ab}$}.
\end{definition}

\begin{definition}\cite[Definition 4.6]{HZZ}
 {\rm Let $\mathcal{W}$ be a class of objects in $\mathcal{C}$. A complex $\mathbf{X}$ is called {\it $\mathcal{C}(-,\mathcal{W})$-exact} (respectively
{\it $\mathcal{C}(\mathcal{W},-)$-exact}) if it is a $\xi$-exact complex
$$\xymatrix@C=2em{\cdots\ar[r]&X_1\ar[r]^{d_1}&X_0\ar[r]^{d_0}&X_{-1}\ar[r]&\cdots}$$ in $\mathcal{C}$ such that  there exists a  $\mathcal{C}(-,\mathcal{W})$-exact (respectively
 $\mathcal{C}(\mathcal{W},-)$-exact) $\mathbb{E}$-triangle $$\xymatrix@C=2em{K_{n+1}\ar[r]^{g_n}&X_n\ar[r]^{f_n}&K_n\ar@{-->}[r]^{\delta_n}&}$$ in $\xi$  and $d_n=g_{n-1}f_n$ for each integer $n$.
}
\end{definition}

\begin{definition} {\rm Let  $A$ be an object in $\mathcal{C}$. A {\it$\mathcal{W}(\xi)$-resolution} of $A$ is a
 $\xi$-exact complex $$\xymatrix@C=2em{\cdots\ar[r]&W_i\ar[r]&\cdots\ar[r]& W_1\ar[r]& W_0\ar[r]& A\ar[r]&0}$$ in $\mathcal{C}$ with all $W_i\in\mathcal{W}$.
A  $\mathcal{W}(\xi)$-resolution of $A$ is called {\it proper $\mathcal{W}(\xi$)-resolution} if it is $\mathcal{C}(\mathcal{W},-)$-exact.
Dually, one can define  the notion of a  {\it $($coproper$)$ $\mathcal{W}(\xi)$-coresolution}}
\end{definition}

%
%

\section{\bf Proper resolutions and coproper coresolutions}
In this section, we provide a method for constructing a proper $\mathcal{W}(\xi)$-resolution (respectively, coproper
$\mathcal{W}(\xi)$-coresolution) of one term in an $\mathbb{E}$-triangle in $\xi$ from those of the other two
terms. At first, we need the following easy observations.

\begin{lem}\label{lem3} Let $\mathcal{W}$ be a class of objects in $\mathcal{C}$.

{\rm (1)} Consider the following commutative diagram of $\mathbb{E}$-triangles in $\xi$
 $$\xymatrix{A\ar@{=}[d]\ar[r]^x&B\ar[d]^b\ar[r]^y&C\ar@{-->}[r]^{c^*\delta'}\ar[d]^c&\\
 A\ar[r]^{x'}&B'\ar[r]^{y'}&C'\ar@{-->}[r]^{\delta'}&.}$$
 If the second row is $\mathcal{C}(\mathcal{W},-)$-exact, then so is the first row.

{\rm (2)} Consider the following commutative diagram of $\mathbb{E}$-triangles in $\xi$
 $$\xymatrix{A\ar[d]^a\ar[r]^x&B\ar[d]^b\ar[r]^y&C\ar@{-->}[r]^{\delta}\ar@{=}[d]&\\
 A'\ar[r]^{x'}&B'\ar[r]^{y'}&C\ar@{-->}[r]^{a_*\delta}&.}$$
If the first row is $\mathcal{C}(-,\mathcal{W})$-exact, then so is the second row.
\end{lem}
\begin{proof}We only prove (1), the proof of (2) is similar. Let $W$ be an object in $\mathcal{W}$. By \cite[Corollary 3.12]{NP}, we have the following exact sequence
$$\xymatrix@C=1cm{\mathcal{C}(W, A)\ar[r]^{\mathcal{C}(W, x)}&\mathcal{C}(W, B)\ar[r]^{\mathcal{C}(W, y)}&\mathcal{C}(W, C)\ar[r]^{c^{*}\delta'_\sharp}&\mathbb{E}(W, A),}$$
where ${c^{*}\delta'_\sharp}(f)=f^{*}c^{*}\delta'$ for any $f\in{\mathcal{C}(W, C)}$.  Note that the sequence $$\xymatrix@C=2em{0\ar[r]&\mathcal{C}(W,A)\ar[r]^{\mathcal{C}(W,x')}&\mathcal{C}(W,B')\ar[r]^{\mathcal{C}(W,y')}&\mathcal{C}(W,C')\ar[r]&0}$$ is exact in ${\rm Ab}$ by hypothesis. Therefore, for any morphism $f\in{\mathcal{C}(W, C)}$, there exists a morphism $g\in{\mathcal{C}(W, B')}$ such that $cf=y'g$. Thus ${c^{*}\delta'_\sharp}(f)=f^{*}c^{*}\delta'=g^{*}y'^{*}\delta'=0$, and hence $\mathcal{C}(W, y)$ is epic. Since $\mathcal{C}(W, x')=\mathcal{C}(W, b)\mathcal{C}(W, x)$ is monic, so is $\mathcal{C}(W, x)$. This implies that $\xymatrix@C=2em{0\ar[r]&\mathcal{C}(W,A)\ar[r]^{\mathcal{C}(W,x)}&\mathcal{C}(W,B)\ar[r]^{\mathcal{C}(W,y)}&\mathcal{C}(W,C)\ar[r]&0}$ is exact in ${\rm Ab}$, as desired.
\end{proof}

Similar to the proof of \cite[Lemma 2.5]{Hua} and its dual, we have the following.
\begin{lem}\label{lem4} { Consider} the  morphism of $\mathbb{E}$-triangles $$\xymatrix{A\ar[d]^a\ar[r]^x&B\ar[d]^b\ar[r]^y&C\ar@{-->}[r]^{\delta}\ar[d]^c&\\
 A'\ar[r]^{x'}&B'\ar[r]^{y'}&C'\ar@{-->}[r]^{\delta'}&.}$$
 { Then the following hold:}
\begin{itemize}
\item[{\rm (1)}] If all morphisms of $\mathcal{C}(a,W)$,  $\mathcal{C}(c,W)$ and $\mathcal{C}(x',W)$ are epic for an object $W$ in $\mathcal{C}$, then so is $\mathcal{C}(b,W)$.

\item[{\rm (2)}] If all morphisms of $\mathcal{C}(a,W)$,  $\mathcal{C}(c,W)$ and $\mathcal{C}(y,W)$ are monic for an object $W$ in $\mathcal{C}$, then so is $\mathcal{C}(b,W)$.

\item[{\rm (3)}] If all morphisms of $\mathcal{C}(W,a)$,  $\mathcal{C}(W,c)$ and $\mathcal{C}(W,y)$ are epic for an object $W$ in $\mathcal{C}$, then so is $\mathcal{C}(W,b)$.

\item[{\rm (4)}] If all morphisms of $\mathcal{C}(W,a)$,  $\mathcal{C}(W,c)$ and $\mathcal{C}(W,x')$ are monic for an object $W$ in $\mathcal{C}$, then so is $\mathcal{C}(W,b)$.
\end{itemize}
\end{lem}

\begin{prop}\label{lem8}
{\rm (1)} If $\xymatrix{X\ar[r]^f&Y\ar[r]^{f'}&Z\ar@{-->}[r]^{\delta}&}$ and $\xymatrix{K\ar[r]^g&L\ar[r]^{g'}&Y\ar@{-->}[r]^{\delta'}&}$ are both $\mathcal{C}(\mathcal{W},-)$-exact and $\mathcal{C}(-,\mathcal{W})$-exact $\mathbb{E}$-triangles in $\xi$, then we have the following commutative diagram
$$\xymatrix{
   K \ar[d]^d \ar@{=}[r] & K \ar[d]^g&\\
  N\ar[d]^e \ar[r]^h & L \ar[d]^{g'} \ar[r]^{h'} & Z \ar@{=}[d]\ar@{-->}[r]^{\delta''}& \\
  X \ar[r]^f\ar@{-->}[d]^{f^{*}\delta'} & Y \ar[r]^{f'}\ar@{-->}[d]^{\delta'} & Z\ar@{-->}[r]^{\delta=e_*\delta''}&\\
  &&&  }
$$
where all rows and columns are both $\mathcal{C}(\mathcal{W},-)$-exact and $\mathcal{C}(-,\mathcal{W})$-exact $\mathbb{E}$-triangles in $\xi$.

{\rm (2)} If $\xymatrix{X\ar[r]^f&Y\ar[r]^{f'}&Z\ar@{-->}[r]^\delta&}$ and $\xymatrix{Y\ar[r]^g&L\ar[r]^{g'}&K\ar@{-->}[r]^{\delta'}&}$ are both $\mathcal{C}(\mathcal{W},-)$-exact and $\mathcal{C}(-,\mathcal{W})$-exact $\mathbb{E}$-triangles in $\xi$, then we have the following commutative diagram
$$\xymatrix{
   X \ar@{=}[d] \ar[r]^f & Y \ar[r]^{f'}\ar[d]^g&Z\ar[d]^d\ar@{-->}[r]^{\delta}&\\
  X\ar[r]^h & L \ar[d]^{g'} \ar[r]^{h'} & N \ar[d]^{e}\ar@{-->}[r]^{\delta''}& \\
  & K\ar@{=}[r]\ar@{-->}[d]^{\delta'} & K\ar@{-->}[d]^{f'_{*}\delta'} &\\
  &&&  }
$$
where all rows and columns are both $\mathcal{C}(\mathcal{W},-)$-exact and $\mathcal{C}(-,\mathcal{W})$-exact $\mathbb{E}$-triangles in $\xi$.
\end{prop}
\begin{proof} (1) It follows from \cite[Theorem 3.2]{HZZ} that we have the desired commutative diagram where all rows and columns are $\mathbb{E}$-triangles in $\xi$. Since the $\mathbb{E}$-triangle $\xymatrix{K\ar[r]^g&L\ar[r]^{g'}&Y\ar@{-->}[r]^{\delta'}&}$ is a $\mathcal{C}(\mathcal{W},-)$-exact in $\xi$,  the first column  in this diagram is $\mathcal{C}(\mathcal{W},-)$-exact by Lemma \ref{lem3}(1). It is easy to check that the second row is $\mathcal{C}(\mathcal{W},-)$-exact by (3) and (4) in Lemma \ref{lem4}. Then all rows and columns in this diagram are $\mathcal{C}(\mathcal{W},-)$-exact. Note that $\mathbb{E}$-triangles $\xymatrix{X\ar[r]^f&Y\ar[r]^{f'}&Z\ar@{-->}[r]^{\delta}&}$ and $\xymatrix{K\ar[r]^g&L\ar[r]^{g'}&Y\ar@{-->}[r]^{\delta'}&}$ are $\mathcal{C}(-,\mathcal{W})$-exact, it is easy to check that $\mathcal{C}(h, W)$ is epic  by Lemma \ref{lem4}(1) for any object $W\in\mathcal{W}$, and $\mathcal{C}(h', W)$ is monic since $\mathcal{C}(h', W)=\mathcal{C}(g', W)\mathcal{C}(f', W)$. This implies that the second row is $\mathcal{C}(-, \mathcal{W})$-exact. It is easy to check that the first column is $\mathcal{C}(-, \mathcal{W})$-exact by $3\times 3$-Lemma.

(2) The proof is dual  of (1).
\end{proof}

\begin{lem}\label{lem5} Let $\xymatrix{A\ar[r]^{x}&B\ar[r]^{y}&C\ar@{-->}[r]^{\delta}&}$ be a
$\mathcal{C}(\mathcal{W},-)$-exact $\mathbb{E}$-triangle in $\xi$.  If both
$$\xymatrix{K_1^A\ar[r]^{g_0^A}&W_0^A\ar[r]^{f_0^A}&A\ar@{-->}[r]^{\delta_0^A}&}~\emph{and} ~\xymatrix{K_1^C\ar[r]^{g_0^C}&W_0^C\ar[r]^{f_0^C}&C\ar@{-->}[r]^{\delta_0^C}&}$$ are $\mathbb{E}$-triangles in $\xi$  with $W_0^C\in\mathcal{W}$, then we have the following commutative diagram:
$$\xymatrix{K_1^A\ar@{-->}[r]^{x_1}\ar[d]_{g_{0}^A}&K_1^B\ar@{-->}[r]^{y_1}\ar@{-->}[d]^{g_{0}^B}&K_1^C\ar[d]^{g_{0}^C}\ar@{-->}[r]^{\delta_1}&\\
W_{0}^A\ar[d]_{f_{0}^A}\ar[r]^{\tiny\begin{pmatrix}1\\0\end{pmatrix}\ \ \ \ \ }&W_{0}^A\oplus W_0^C\ar@{-->}[d]^{f_{0}^B}\ar[r]^{\tiny\ \ \ \ \ \begin{pmatrix}0&1\end{pmatrix}}&
W_0^C\ar[d]^{f_{0}^C}\ar@{-->}[r]^0&\\
A\ar[r]^{x}\ar@{-->}[d]^{\delta_0^A}&B\ar[r]^{y}\ar@{-->}[d]^{\delta_0^B}&C\ar@{-->}[r]^{\delta}\ar@{-->}[d]^{\delta_0^C}&\\
&&&
}$$
{ where} all rows and columns are $\mathbb{E}$-triangles in $\xi$. Moreover,

{\rm (1)} If the first and the third columns in this diagram are $\mathcal{C}(\mathcal{W},-)$-exact, then so are all $\mathbb{E}$-triangles in this diagram.

{\rm (2)} If the third row, the first and  the third columns in this diagram are $\mathcal{C}(-,\mathcal{W})$-exact, then so are all $\mathbb{E}$-triangles in this diagram.
\end{lem}

{\bf Proof.} Since $\xymatrix{A\ar[r]^{x}&B\ar[r]^{y}&C\ar@{-->}[r]^{\delta}&}$ is a
$\mathcal{C}(\mathcal{W},-)$-exact, there exists $z\in \mathcal{C}(W_0^C, B)$ such that $f_0^C=yz$. So $(f_0^C)^*\delta=z^*y^*\delta=0$, and there exists a $\xi$-deflation $f_0^B\in\mathcal{C}(W_0^A\oplus W_0^C, B)$ by Proposition \ref{prop} which makes the following diagram commutative:
$$\xymatrix{W_{0}^A\ar[d]_{f_{0}^A}\ar[r]^{\tiny\begin{pmatrix}1\\0\end{pmatrix}\ \ \ \ \ }&W_{0}^A\oplus W_0^C\ar@{-->}[d]^{f_{0}^B}\ar[r]^{\tiny\ \ \ \ \ \begin{pmatrix}0&1\end{pmatrix}}&
W_0^C\ar[d]^{f_{0}^C}\ar@{-->}[r]^0&\\
A\ar[r]^{x}&B\ar[r]^{y}&C\ar@{-->}[r]^{\delta}&.
}$$
 Assume that $\xymatrix{K_1^B\ar[r]^{g_0^B\  \ \ }&W_0^A\oplus W_0^C\ar[r]^{\ \ \ \ \ f_0^B}&B\ar@{-->}[r]^{\delta_0^B}&}$ is an $\mathbb{E}$-triangle in $\xi$, then we have the following commutative diagram:
$$\xymatrix{K_1^A\ar@{-->}[r]^{x_1}\ar[d]_{g_{0}^A}&K_1^B\ar@{-->}[r]^{y_1}\ar[d]^{g_{0}^B}&K_1^C\ar[d]^{g_{0}^C}\ar@{-->}[r]^{\delta_1}&\\
W_{0}^A\ar[d]_{f_{0}^A}\ar[r]^{\tiny\begin{pmatrix}1\\0\end{pmatrix}\ \ \ \ \ }&W_{0}^A\oplus W_0^C\ar[d]^{f_{0}^B}\ar[r]^{\tiny\ \ \ \ \ \begin{pmatrix}0&1\end{pmatrix}}&
W_0^C\ar[d]^{f_{0}^C}\ar@{-->}[r]^0&\\
A\ar[r]^{x}\ar@{-->}[d]^{\delta_0^A}&B\ar[r]^{y}\ar@{-->}[d]^{\delta_0^B}&C\ar@{-->}[r]^{\delta}\ar@{-->}[d]^{\delta_0^C}&\\
&&&
}$$
where all rows and columns are $\mathbb{E}$-triangles in $\xi$ by \cite[Lemma 4.14]{HZZ}.
It is easy to check that the $\mathbb{E}$-triangle $\xymatrix@C=0.8cm{K_1^B\ar[r]^{g_0^B\quad\;\;}&W_0^A\oplus W_0^C\ar[r]^{\qquad f_0^B}&B\ar@{-->}[r]^{\delta_0^B}&}$ is $\mathcal{C}(\mathcal{W},-)$-exact by (3) and (4) in  Lemma \ref{lem4}. Hence, the first row is $\mathcal{C}(\mathcal{W},-)$-exact by $3\times 3$-Lemma. Similarly, one can prove the result of (2). This completes the proof.
\qed

\begin{lem}\label{lem6} Let $\xymatrix{A\ar[r]^{x}&B\ar[r]^{y}&C\ar@{-->}[r]^{\delta}&}$ be a
$\mathcal{C}(-,\mathcal{W})$-exact $\mathbb{E}$-triangle in $\xi$.  If  both
$$\xymatrix{A\ar[r]^{g^{0}_A}&W^{0}_A\ar[r]^{f^{0}_A}&K_{-1}^A\ar@{-->}[r]^{\delta^{0}_A}&}~\textrm{and}~
\xymatrix{C\ar[r]^{g^{0}_C}&W^{0}_C\ar[r]^{f^{0}_C}&K^{1}_C\ar@{-->}[r]^{\delta^{0}_C}&}$$ are $\mathbb{E}$-triangles in $\xi$  with $W^{0}_A\in\mathcal{W}$,
then we have the following commutative diagram:
$$\xymatrix{A\ar[r]^{x}\ar[d]_{g^{0}_A}&B\ar[r]^{y}\ar[d]^{g^{0}_B}&C\ar[d]^{g^{0}_C}\ar@{-->}[r]^{\delta}&\\
W^{0}_A\ar[d]_{f^{0}_A}\ar[r]^{\tiny\begin{pmatrix}1\\0\end{pmatrix}\ \ \ \ \ \ }&W^{0}_A\oplus W^{0}_C\ar[d]^{f^{0}_B}
\ar[r]^{\tiny\ \ \ \ \ \ \begin{pmatrix}0&1\end{pmatrix}}&
W^0_C\ar[d]^{f^{0}_C}\ar@{-->}[r]^0&\\
K^{1}_A\ar@{-->}[r]^{x^{1}}\ar@{-->}[d]_{\delta^{0}_A}&K^{1}_B\ar@{-->}[r]^{y^{1}}\ar@{-->}[d]^{\delta^{1}_B}&K^{1}_C\ar@{-->}[d]^{\delta^{1}_C}
\ar@{-->}[r]^{\delta^{1}}&\\
&&&}$$
{ where} all rows and columns are $\mathbb{E}$-triangles in $\xi$. Moreover,

{\rm (1)}  If the first and  the  third columns in this diagram are $\mathcal{C}(-,\mathcal{W})$-exact, then so are all $\mathbb{E}$-triangles in this diagram.

{\rm (2)} If the first row, the first and the third columns in this diagram are $\mathcal{C}(\mathcal{W},-)$-exact, then so are all $\mathbb{E}$-triangles in this diagram.

\end{lem}
\begin{proof} The proof is dual to Lemma \ref{lem5}.
\end{proof}

The next result provides a method for constructing a proper $\mathcal{W}(\xi)$-
resolution of the first term in an $\mathbb{E}$-triangle in $\xi$ from those of the last two
terms.

\begin{thm}\label{thm1} Let \begin{equation}\label{equ1}\xymatrix{X\ar[r]^x&X^0\ar[r]^y&X^1\ar@{-->}[r]^{\delta}&}\end{equation} be an $\mathbb{E}$-triangle in $\xi$.
Assume that $\mathcal{W}$ is closed
under finite direct sums and hokernel of $\xi$-deflations, and let
 \begin{equation}\label{equ2} \xymatrix{\cdots\ar[r]&W_i^0\ar[r]^{d_i^0}&\cdots\ar[r]&W_1^0\ar[r]^{d_1^0}&W_0^0\ar[r]^{d_0^0}&X^0\ar[r]&0}~~\textrm{and}\end{equation}
  \begin{equation}\label{equ3}\xymatrix{\cdots\ar[r]&W_i^1\ar[r]^{d_i^1}&\cdots\ar[r]&W_1^1\ar[r]^{d_1^1}&W_0^1\ar[r]^{d_0^1}&X^1\ar[r]&0}\end{equation}
 be  proper $\mathcal{W}(\xi)$-resolutions of $X^0$ and $X^1$, respectively.

{\rm (1)} Then we obtain the following proper $\mathcal{W}(\xi)$-resolution of $X$

\begin{equation}\label{equ4} \xymatrix{\cdots\ar[r]&W_{i+1}^1\oplus W_i^0\ar[r]&\cdots\ar[r]&W_2^1\oplus W_1^0\ar[r]&W\ar[r]&X\ar[r]&0}\end{equation}
 and  an $\mathbb{E}$-triangle
 \begin{equation}\label{equ5} \xymatrix{W\ar[r]&W_1^1\oplus W_0^0\ar[r]&W^1_0\ar@{-->}[r]&}\end{equation}
in $\xi$.

{\rm (2)} If both  the $\xi$-exact complexes (\ref{equ2}), (\ref{equ3}) and the $\mathbb{E}$-triangle (\ref{equ1}) are $\mathcal{C}(-, \mathcal{W})$-exact, then so is (\ref{equ4}).
\end{thm}
{\bf Proof.} (1) Since (\ref{equ2}) and (\ref{equ3}) are proper $\mathcal{W}(\xi)$-resolutions of $X^0$ and $X^1$ respectively, there exist $\mathcal{C}(\mathcal{W},-)$-exact $\mathbb{E}$-triangles
$$\xymatrix{K^0_{i+1}\ar[r]^{g_i^0}&W_i^0\ar[r]^{f_i^0}&K_i^0\ar@{-->}[r]^{\delta_i^0}&}~\textrm{and}
~\xymatrix{K^1_{i+1}\ar[r]^{g_i^1}&W_i^1\ar[r]^{f_i^1}&K_i^1\ar@{-->}[r]^{\delta_i^1}&}$$ in $\xi$ for each integer $i\geqslant 0$
 with $K_0^0=X^0$ and $K_0^1=X^1$. It follows from \cite[Theorem 3.2]{HZZ} that there exists following commutative diagram:
$$\xymatrix{
   & K_1^1 \ar[d]^{m_1} \ar@{=}[r] & K_1^1 \ar[d]^{g_0^1}&\\
  X \ar@{=}[d] \ar[r]^{m_2} & M \ar[d]^{e_1} \ar[r]^{e_2} & W_0^1 \ar[d]^{f_0^1}\ar@{-->}[r]^{\eta_2}&&(\ast) \\
 X \ar[r]^x & X^0 \ar[r]^y\ar@{-->}[d]^{\eta_1} & X^1\ar@{-->}[r]^{\delta}\ar@{-->}[d]^{\delta_0^1}&&\\
  &&&&}
$$
 { where} all rows and columns are $\mathbb{E}$-triangles in $\xi$. Because the third column in diagram $(\ast)$ is $\mathcal{C}(\mathcal{W},-)$-exact $\mathbb{E}$-triangle, so is the middle column by Lemma \ref{lem3}(1). Thus by Lemma \ref{lem5}(1), we get the following commutative diagram:
$$\xymatrix{
  K_2^1 \ar[d]^{g_1^1} \ar[r] & L_1 \ar[d] \ar[r] & K_1^0 \ar[d]^{g_0^0}\ar@{-->}[r]& \\
  W_1^1\ar[d]^{f_1^1} \ar[r] & W_1^1\oplus W_0^0 \ar[d] \ar[r] & W_0^0 \ar[d]^{f_0^0}\ar@{-->}[r]^0& &(\ast\ast)\\
  K_1^1 \ar[r]^{m_1}\ar@{-->}[d]^{\delta_1^1} & M \ar[r]^{e_1}\ar@{-->}[d]^{\alpha_1} & X^0\ar@{-->}[d]^{\delta_0^0} \ar@{-->}[r]^{\eta_1}& \\
  &&& }
$$
where all rows and columns are $\mathcal{C}(\mathcal{W},-)$-exact $\mathbb{E}$-triangles in $\xi$ and the middle row is split. It is clear  $W_1^1\oplus W_0^0\in\mathcal{W}$ by assumption.

On the other hand, we have the following commutative diagram:
$$\xymatrix{
    L_1 \ar[d] \ar@{=}[r] & L_1 \ar[d]&\\
  W \ar[d] \ar[r] & W_1^1\oplus W_0^0\ar[d] \ar[r] & W_0^1 \ar@{=}[d]\ar@{-->}[r]&&(\star) \\
  X \ar[r]\ar@{-->}[d] & M \ar[r]\ar@{-->}[d] & W_0^1\ar@{-->}[r]& \\
  &&& }
$$
where all rows and columns are $\mathbb{E}$-triangles in $\xi$ by \cite[Theorem 3.2]{HZZ} and the second row is the desired $\mathbb{E}$-triangle (\ref{equ5}). It is easy to see that $W\in\mathcal{W}$ because $\mathcal{W}$ is closed under hokernel of $\xi$-deflations.
Since the middle column is $\mathcal{C}(\mathcal{W},-)$-exact,  so is the first column by Lemma \ref{lem3}(1).

By Lemma \ref{lem5}(1) again, we we get the following commutative diagram:
$$\xymatrix{
  K_3^1 \ar[d] \ar[r] & L_2 \ar[d] \ar[r] & K_2^0 \ar[d]\ar@{-->}[r]& \\
  W_2^1\ar[d] \ar[r] & W_2^1\oplus W_1^0 \ar[d] \ar[r] & W_1^0 \ar[d]\ar@{-->}[r]& &(\star\star)\\
  K_2^1 \ar[r]\ar@{-->}[d] & L_1 \ar[r]\ar@{-->}[d] & K_1^0 \ar@{-->}[r]\ar@{-->}[d]&\\
  &&&  }
$$
where all rows and columns are $\mathcal{C}(\mathcal{W},-)$-exact $\mathbb{E}$-triangles in $\xi$. Continuing in this process, we get the desired $\xi$-exact complex (\ref{equ4}), where $\xymatrix{L_{i+1}\ar[r]&W_{i+1}^0\oplus W_i^0\ar[r]&L_i\ar@{-->}[r]&}$, for each integer $i\geqslant 1$,
 and $\xymatrix{L_{1}\ar[r]&W\ar[r]&X\ar@{-->}[r]&}$ are $\mathcal{C}(\mathcal{W},-)$-exact $\mathbb{E}$-triangles in $\xi$.

(2)  Note that both the third row and the third column in diagram $(\ast)$ are $\mathcal{C}(-,\mathcal{W})$-exact, so the second row and the second column in this  diagram are also $\mathcal{C}(-,\mathcal{W})$-exact. Since both the first and third columns in diagram $(\ast\ast)$
are $\mathcal{C}(-,\mathcal{W})$-exact by assumption,   both the first row and the middle column in this diagram are also
 $\mathcal{C}(-,\mathcal{W})$-exact  by Lemma \ref{lem5}(2). It is easy to check that the first column in diagram $(\star)$ is $\mathcal{C}(-,\mathcal{W})$-exact because the third row and the second column  in this diagram are $\mathcal{C}(-,\mathcal{W})$-exact. Note that the third row,  the first and third columns in diagram $(\star\star)$
 are $\mathcal{C}(-,\mathcal{W})$-exact, so is the  middle column in this diagram by Lemma \ref{lem5}(2).
  Finally we  deduce that the $\xi$-exact complex (\ref{equ4}) is $\mathcal{C}(-,\mathcal{W})$-exact.  \qed
\medskip

%

Dual to Theorem \ref{thm1}, we have the following result which provides a method for constructing a coproper
$\mathcal{W}(\xi)$-coresolution of the last term in an $\mathbb{E}$-triangle in $\xi$ from those of the
first two terms.

\begin{thm}\label{thm2} Let \begin{equation}\label{equ6}\xymatrix{Y_1\ar[r]&Y_0\ar[r]&Y\ar@{-->}[r]&}\end{equation} be an $\mathbb{E}$-triangle in $\xi$. Assume that $\mathcal{W}$ is closed
 under finite direct sums and hocokernel of $\xi$-inflations, and let
 \begin{equation}\label{equ7} \xymatrix{0\ar[r]&Y_0\ar[r]&W_0^0\ar[r]&W^1_0\ar[r]&\cdots\ar[r]&W^i_0\ar[r]&\cdots}\end{equation}
 and
  \begin{equation}\label{equ8}\xymatrix{0\ar[r]&Y_1\ar[r]&W^0_1\ar[r]&W_1^1\ar[r]&\cdots\ar[r]&W^i_1\ar[r]&\cdots}\end{equation}
 be proper  $\mathcal{W}(\xi)$-coresolutions of $Y_0$ and $Y_1$, respectively.

 {\rm (1)} Then we get the following coproper $\mathcal{W}(\xi)$-resolution of $Y$:
\begin{equation}\label{equ9} \xymatrix{0\ar[r]&Y\ar[r]&W\ar[r]&W^1_0\oplus W^2_1\ar[r]&\cdots\ar[r]&W^{i}_0\oplus W^{i+1}_1\ar[r]&\cdots}\end{equation}
 and  an $\mathbb{E}$-triangle
 \begin{equation}\label{equ10} \xymatrix{W_1^0\ar[r]&W_0^0\oplus W_1^1\ar[r]&W\ar@{-->}[r]&}\end{equation}
in $\xi$.

{\rm (2)} If both the $\xi$-exact complexes (\ref{equ7}), (\ref{equ8}) and the $\mathbb{E}$-triangle (\ref{equ6})  are $\mathcal{C}(\mathcal{W},-)$-exact, then so is (\ref{equ9}).
\end{thm}
{\bf Proof.} The proof is dual to Theorem \ref{thm1}, so we omit it.  \qed
\medskip

%

The next result provides a method for constructing a proper $\mathcal{W}(\xi)$-resolution
of the last term in an $\mathbb{E}$-triangle in $\xi$ from those of the first two
terms.

\begin{thm}\label{thm3} Let \begin{equation}\label{equ11}\xymatrix{X_1\ar[r]&X_0\ar[r]&X\ar@{-->}[r]&}\end{equation} be a $\mathcal{C}(\mathcal{W},-)$-exact $\mathbb{E}$-triangle in $\xi$.
 Assume that $\mathcal{W}$ is closed under finite direct sums, and assume that
 \begin{equation}\label{equ12} \xymatrix{&W^n_0\ar[r]&\cdots\ar[r]&W^1_0\ar[r]&W_0^0\ar[r]&X_0\ar[r]&0}\end{equation}
 and
  \begin{equation}\label{equ13}\xymatrix{&W^{n-1}_1\ar[r]&\cdots\ar[r]&W_1^1\ar[r]&W^0_1\ar[r]&X_1\ar[r]&0}\end{equation}
 are proper $\mathcal{W}(\xi)$-resolutions of $X_0$ and $X_1$, respectively.

 {\rm (1)} Then we get a proper $\mathcal{W}(\xi)$-resolution of $X$
 \begin{equation}\label{equ14} \xymatrix{W^n_0\oplus W_1^{n-1}\ar[r]&\cdots\ar[r]&W_0^2\oplus W_1^1\ar[r]&W_0^1\oplus W_1^0\ar[r]&W_0^0\ar[r]&X\ar[r]&0}\end{equation}

 {\rm (2)} If both the $\xi$-exact complexes (\ref{equ12}), (\ref{equ13}) and the $\mathbb{E}$-triangle (\ref{equ11}) are $\mathcal{C}(-, \mathcal{W})$-exact,
 then so is the $\xi$-exact complex (\ref{equ14}).
\end{thm}
\begin{proof} (1) Since (\ref{equ12}) and (\ref{equ13}) are proper $\mathcal{W}(\xi)$-resolutions of $X_0$ and $X_1$ respectively, there exist $\mathcal{C}(\mathcal{W},-)$-exact $\mathbb{E}$-triangles
$$\xymatrix{K_0^{i+1}\ar[r]&W^i_0\ar[r]&K^i_0\ar@{-->}[r]^{\delta^i_0}&}~\textrm{and} ~\xymatrix{K_1^{i+1}\ar[r]&W^i_1\ar[r]&K^i_1\ar@{-->}[r]^{\delta^i_1}&}$$ in  $\xi$ for each integer $i\geqslant 0$ with $K_0^0=X_0$ and $K^0_1=X_1$.  It follows from \cite[Theorem 3.2]{HZZ} that there exists following commutative diagram
$$\xymatrix{
   K_0^1 \ar[d] \ar@{=}[r] & K_0^1 \ar[d]&\\
  L_1\ar[d] \ar[r] & W_0^0 \ar[d] \ar[r] & X \ar@{=}[d]\ar@{-->}[r]& &(\dagger)\\
  X_1 \ar[r]\ar@{-->}[d] & X_0 \ar[r]\ar@{-->}[d] & X\ar@{-->}[r]&\\
  &&&  }
$$
where all rows and columns are $\mathbb{E}$-triangle in $\xi$. Note that the middle column in diagram $(\dagger)$ is $\mathcal{C}(\mathcal{W},-)$-exact,
so is the first column by Lemma \ref{lem3}(1). Since the third row is $\mathcal{C}(\mathcal{W},-)$-exact by assumption, it is easy to check that the middle
 row   is  $\mathcal{C}(\mathcal{W},-)$-exact. By Lemma \ref{lem5}(1), we get the following commutative diagram:
$$\xymatrix{
  K_0^2 \ar[d] \ar[r] & L_2 \ar[d] \ar[r] & K_1^1 \ar[d]\ar@{-->}[r]& \\
  W_0^1\ar[d] \ar[r] & W_0^1\oplus W_1^0 \ar[d] \ar[r] & W_1^0 \ar[d]\ar@{-->}[r]^0& &(\ddagger)\\
  K_0^1 \ar[r]\ar@{-->}[d] & L_1 \ar[r]\ar@{-->}[d] & X_1 \ar@{-->}[r]\ar@{-->}[d]&\\
  &&&  }
$$
where all rows and columns  are $\mathcal{C}(\mathcal{W},-)$-exact $\mathbb{E}$-triangles in $\xi$. It is clear  $W_0^1\oplus W_1^0\in\mathcal{W}$ by assumption.
Then by using Lemma \ref{lem5}(1) iteratively, we get the desired proper $\mathcal{W}(\xi)$-resolution (\ref{equ14}) of $X$ which is spliced by $\mathcal{C}(\mathcal{W},-)$-exact $\mathbb{E}$-triangles $\xymatrix{L_{1}\ar[r]&W_0^0\ar[r]&X\ar@{-->}[r]&}$ and
 $\xymatrix{L_{i+1}\ar[r]&W^{i}_0\oplus W_1^{i-1}\ar[r]&L_i\ar@{-->}[r]&}$  for each integer $i\geqslant 1$.

(2) { It follows from Proposition \ref{lem8}(1) that all $\mathbb{E}$-triangles in diagram $(\dagger)$ are both $\mathcal{C}(\mathcal{W}-,)$-exact and $\mathcal{C}(-,\mathcal{W})$-exact.}  Both the first and third columns in  diagram $(\ddagger)$ are $\mathcal{C}(-,\mathcal{W})$-exact
  by assumption, so  the middle column in this diagram is also $\mathcal{C}(-,\mathcal{W})$-exact by Lemma \ref{lem5}(2).
   Finally we  deduce that the $\xi$-exact complex (\ref{equ14}) is $\mathcal{C}(-,\mathcal{W})$-exact.
\end{proof}

The next result which is dual to Theorem \ref{thm3}, provides a method for constructing a coproper
$\mathcal{W}(\xi)$-coresolution of the first term in an $\mathbb{E}$-triangle in $\xi$ from those of the
last two terms.

\begin{thm}\label{thm4} Let \begin{equation}\label{equ15}\xymatrix{Y\ar[r]&Y^0\ar[r]&Y^1\ar@{-->}[r]&}\end{equation} be a $\mathcal{C}(-, \mathcal{W})$-exact $\mathbb{E}$-triangle in $\xi$.
Assume that $\mathcal{W}$ is closed under finite direct sums, and let
 \begin{equation}\label{equ16} \xymatrix{0\ar[r]&Y^0\ar[r]&W_0^0\ar[r]&W_1^0\ar[r]&\cdots\ar[r]&W^0_n}\end{equation}
 and
  \begin{equation}\label{equ17}\xymatrix{0\ar[r]&Y^1\ar[r]&W_1^0\ar[r]&W_1^1\ar[r]&\cdots\ar[r]&W_{n-1}^1}\end{equation}
 be coproper $\mathcal{W}(\xi)$-resolutions of $Y^0$ and $Y^1$, respectively.

 {\rm (1)} Then we get a coproper $\mathcal{W}(\xi)$-resolution of $Y$
 \begin{equation}\label{equ18} \xymatrix{0\ar[r]&Y\ar[r]&W_0^0\ar[r]&W_1^0\oplus W_0^1\ar[r]&W_1^1\oplus W_0^2\ar[r]&\cdots\ar[r]&W_1^{n-1}\oplus W_0^n.}
 \end{equation}

 {\rm (2)} { If both the $\xi$-exact complexes (\ref{equ16}) and (\ref{equ17}) and the $\mathbb{E}$-triangle (\ref{equ15}) are $\mathcal{C}(\mathcal{W},-)$-exact, the so is the $\xi$-exact complex (\ref{equ18})}.
\end{thm}
\begin{proof} The proof is dual to Theorem \ref{thm3}, so we omit it.
\end{proof}

\begin{rem}\label{rem:3.11} Note that extriangulated categories are a simultaneous generalization of
abelian categories and triangulated categories. It follows that  Theorems \ref{thm1}, \ref{thm2}, \ref{thm3} and \ref{thm4} here unify Theorems 3.2, 3.4, 3.6 and 3.8 proved by Huang in abelian categories, and Theorems 2.3, 2.5, 2.7 and 2.9 proved by Yang-Wang in triangulated categories. It should be noted that our results here are new for an exact category case.
\end{rem}

\section{\bf Gorensteinness in extriangulated categories}
In this section, some applications of the results in Section 3 are given. In the following, we always assume that $\mathcal{W}$ is a class of objects in $\mathcal{C}$ which is closed under isomorphisms and finite direct sums. We introduce
the Gorenstein category $\mathcal{GW}(\xi)$ in extriangulated categories and  demonstrate that this category shares some
basic properties with Gorenstein category in the abelian category or in the triangulated category.

\begin{definition} {\rm Let $X$ be an object of $\mathcal{C}$. A {\it complete $\mathcal{W}(\xi)$-resolution} of $X$ is both $\mathcal{C}(\mathcal{W},-)$-exact
 and $\mathcal{C}(-,\mathcal{W})$-exact $\xi$-exact complex$$\xymatrix{\cdots\ar[r]&W_1\ar[r]&W_0\ar[r]&W^0\ar[r]&W^1\ar[r]&\cdots}$$ in $\mathcal{W}$
such that $\xymatrix{X_1\ar[r]&W_0\ar[r]&X\ar@{-->}[r]&}$ and $\xymatrix{X\ar[r]&W^0\ar[r]&X^1\ar@{-->}[r]&}$  are corresponding $\mathbb{E}$-triangles in $\xi$.}
\end{definition}

The {\it Gorenstein subcategory $\mathcal{GW}(\xi)$} of $\mathcal{C}$ is defined as
\begin{center} $\mathcal{GW}(\xi)=\{X\in\mathcal{C}~|~X$ admits a complete $\mathcal{W}(\xi)$-resolution$\}$.\end{center}
Set $\mathcal{GW}^1(\xi)=\mathcal{GW}(\xi)$, and inductively set $\mathcal{GW}^{n+1}(\xi)=\mathcal{G}(\mathcal{GW}^n(\xi))$ for any $n\geq 1$.

\begin{remark}\label{rem:4.2} {\emph{(1)} Assume that $\mathcal{C}$ is an abelian category. If $\xi$ is the class of exact sequences and $\mathcal{W}$ is a full additive subcategory of $\mathcal{C}$ that is closed under isomorphisms, then Gorenstein subcategory $\mathcal{GW}(\xi)$ defined
in here coincides with the earlier one given by Sather-Wagstaff, Sharif and White in \cite{Sather}.

\emph{(2)} Assume that $(\mathcal{T},\Sigma, \Delta)$ is a triangulated category and $\xi$ is { a proper class  of triangles} which is closed under suspension \emph{(see \cite[Section 2.2]{Bel1})}. If $\mathcal{W}$ is an additive full subcategory of $\mathcal{T}$ closed under isomorphisms and $\Sigma$-stable, i.e., $\Sigma(\mathcal{W})=\mathcal{W}$, then Gorenstein subcategory $\mathcal{GW}(\xi)$ defined
in here coincides with the earlier one given by Yang and Wang in \cite{Yang3}.

\emph{(3)} Assume that the extriangulated category $(\mathcal{C}, \mathbb{E}, \mathfrak{s})$ has enough $\xi$-projectives and enough $\xi$-injectives. If
$\mathcal{W}$ is the class of $\xi$-projectives \emph{(}resp. $\xi$-injectives\emph{)}, then Gorenstein subcategory $\mathcal{GW}(\xi)$ defined
in here coincides with the subcategory of $\mathcal{C}$ consisting of $\xi$-$\mathcal{G}$projective \emph{(}resp., $\xi$-$\mathcal{G}$injective objects\emph{) (}see \cite{HZZ}\emph{)}.}
\end{remark}

\begin{lem}\label{lem10} Given both $\mathcal{C}(\mathcal{W},-)$-exact and $\mathcal{C}(-,\mathcal{W})$-exact $\mathbb{E}$-triangle
$$\xymatrix{X\ar[r]&Y\ar[r]&Z\ar@{-->}[r]&}$$ in $\xi$, if $X, Z\in\mathcal{GW}(\xi)$ , then so is $Y$.
\end{lem}
\begin{proof} Assume that $X,Z\in\mathcal{GW}(\xi)$, there exist complete $\mathcal{W}(\xi)$-resolutions $$\xymatrix{\cdots\ar[r]&W_1\ar[r]&W_0\ar[r]&W^0\ar[r]&W^1\ar[r]&\cdots}$$ and $$\xymatrix{\cdots\ar[r]&V_1\ar[r]&V_0\ar[r]&V^0\ar[r]&V^1\ar[r]&\cdots}$$
of $X$ and $Z$ respectively. It is straightforward to show that $$\xymatrix{\cdots\ar[r]&W_1\oplus V_1\ar[r]&W_0\oplus V_0\ar[r]&W^0\oplus V^0\ar[r]&W^1\oplus V^1\ar[r]&\cdots}$$ is a complete $\mathcal{W}(\xi)$-resolution of $Y$ by repeating application of {Lemmas \ref{lem5} and \ref{lem6}}.
\end{proof}
As a main application of the results in Section 3, we have the main result of this section.
\begin{thm} \label{thm:stability} The following are true for any extriangulated category $(\mathcal{C}, \mathbb{E}, \mathfrak{s})$:

\emph{(1)}  $\mathcal{GW}^n(\xi)=\mathcal{GW}(\xi)$ for any $n\geqslant1$;

\emph{(2)}  $\mathcal{GW}(\xi)$ is closed under direct summands.
\end{thm}

\begin{proof} (1) It suffices to show that $\mathcal{GW}^2(\xi)=\mathcal{GW}(\xi)$. Let $G\in\mathcal{GW}(\xi)$. Note that  $$\xymatrix{G\ar[r]^1&G\ar[r]&0\ar@{-->}[r]&}~\textrm{and } \xymatrix{0\ar[r]&G\ar[r]^1&G\ar@{-->}[r]&}$$  are  $\mathbb{E}$-triangles in $\xi$. It is easy to check that $$\xymatrix{\cdots\ar[r]&0\ar[r]&G\ar[r]^1&G\ar[r]&0\ar[r]&\cdots}$$ is a complete $\mathcal{GW}(\xi)$-resolution of $G$, then $G\in\mathcal{GW}^2(\xi)$ which implies that $\mathcal{GW}(\xi)\subseteq\mathcal{GW}^2(\xi)$.

Let $X$ be an object in $\mathcal{GW}^2(\xi)$ and $$\xymatrix{\cdots\ar[r]&G_1\ar[r]&G_0\ar[r]&G^0\ar[r]&G^1\ar[r]&\cdots}$$ a complete $\mathcal{GW}(\xi)$-resolution of $X$. That is, there exist both $\mathcal{C}(\mathcal{GW}(\xi),-)$-exact and $\mathcal{C}(-,\mathcal{GW}(\xi))$-exact $\mathbb{E}$-triangles
$\xymatrix{K_{i+1}\ar[r]&G_i\ar[r]&K_i\ar@{-->}[r]&}$  and $\xymatrix{K^{i}\ar[r]&G^i\ar[r]&K^{i+1}\ar@{-->}[r]&}$ in $\xi$ with $G_i,G^i\in\mathcal{GW}(\xi)$ for any $i\geq 0$, where $K_0=K^0=X$. Since $G_0\in\mathcal{GW}(\xi)$, there exists a both $\mathcal{C}(\mathcal{W},-)$-exact and $\mathcal{C}(-,\mathcal{W})$-exact $\mathbb{E}$-triangle
$\xymatrix{L_1\ar[r]&W_0\ar[r]&G_0\ar@{-->}[r]&}$ in $\xi$ with $W_0\in\mathcal{W}$ and $L_1\in\mathcal{GW}(\xi)$. Note that the $\mathbb{E}$-triangle
$\xymatrix{K_{1}\ar[r]&G_0\ar[r]&X\ar@{-->}[r]&}$ in $\xi$ is both $\mathcal{C}(\mathcal{W},-)$-exact and $\mathcal{C}(-,\mathcal{W})$-exact, so we have the following commutative diagram
$$\xymatrix{
   L_1\ar[d] \ar@{=}[r] & L_1 \ar[d]&\\
  M_1\ar[d] \ar[r] & W_0 \ar[d] \ar[r] & X \ar@{=}[d]\ar@{-->}[r]& \\
  K_1 \ar[r]\ar@{-->}[d] & G_0 \ar[r]\ar@{-->}[d] & X\ar@{-->}[r]&\\
  &&&  }
$$
where all rows and columns are both $\mathcal{C}(\mathcal{W},-)$-exact and $\mathcal{C}(-,\mathcal{W})$-exact $\mathbb{E}$-triangles in $\xi$ by Proposition \ref{lem8}(1). It follows from \cite[Theorem 3.2]{HZZ} that we have the following commutative diagram
$$\xymatrix{
   &K_2 \ar[d] \ar@{=}[r] & K_2 \ar[d]\\
  L_1\ar@{=}[d] \ar[r] & N_1\ar[d] \ar[r] & G_1 \ar[d]\ar@{-->}[r]& &(\diamond)\\
  L_1\ar[r] & M_1 \ar[r]\ar@{-->}[d] & K_1\ar@{-->}[r]\ar@{-->}[d]&\\
  &&&  }
$$
where all rows and columns are  $\mathbb{E}$-triangles in $\xi$. Moreover, all $\mathbb{E}$-triangles in diagram $(\diamond)$ are $\mathcal{C}(\mathcal{W},-)$-exact by Lemma \ref{lem3}(1).

Next we claim that the $\mathbb{E}$-triangle $\xymatrix{L_{1}\ar[r]&N_1\ar[r]&G_1\ar@{-->}[r]&}$ in $\xi$ is $\mathcal{C}(-, \mathcal{W})$-exact. In fact, there exists a both $\mathcal{C}(\mathcal{W},-)$-exact and $\mathcal{C}(-,\mathcal{W})$-exact $\mathbb{E}$-triangle
$\xymatrix{H\ar[r]&W\ar[r]&G_1\ar@{-->}[r]&}$ in $\xi$ with $W\in\mathcal{W}$ and $H\in\mathcal{GW}(\xi)$ since $G_1\in\mathcal{GW}(\xi)$. So we have the following commutative diagram
$$\xymatrix{
   &H \ar[d] \ar@{=}[r] & H \ar[d]\\
  L_1\ar@{=}[d] \ar[r] & Z\ar[d] \ar[r] & W \ar[d]\ar@{-->}[r]& &(\diamond\diamond)\\
  L_1\ar[r] & N_1 \ar[r]\ar@{-->}[d] & G_1\ar@{-->}[r]\ar@{-->}[d]&\\
  &&&  }
$$
where all rows and columns are  $\mathbb{E}$-triangles in $\xi$. Moreover, the second row in diagram  $(\diamond\diamond)$ is $\mathcal{C}(\mathcal{W},-)$-exact by Lemma \ref{lem3}(1), and it is split as $W\in\mathcal{W}$. Hence the second row in diagram  $(\diamond\diamond)$ is $\mathcal{C}(-, \mathcal{W})$-exact, it follows from \cite[Lemma 4.10(1)]{HZZ} that the $\mathbb{E}$-triangle $\xymatrix{L_{1}\ar[r]&N_1\ar[r]&G_1\ar@{-->}[r]&}$ in $\xi$ is $\mathcal{C}(-, \mathcal{W})$-exact because the third column in diagram  $(\diamond\diamond)$ is $\mathcal{C}(-,\mathcal{W})$-exact.
Since $G_1, L_1\in\mathcal{GW}(\xi)$, $N_1\in\mathcal{GW}(\xi)$ by Lemma \ref{lem10}.  It is easy to show that all $\mathbb{E}$-triangles in diagram $(\diamond)$ are $\mathcal{C}(-,\mathcal{W})$-exact. Hence all $\mathbb{E}$-triangles in diagram $(\diamond)$ are both $\mathcal{C}(\mathcal{W},-)$-exact and $\mathcal{C}(-,\mathcal{W})$-exact. Since $N_1\in\mathcal{GW}(\xi)$, there exists a both $\mathcal{C}(\mathcal{W},-)$-exact and $\mathcal{C}(-,\mathcal{W})$-exact $\mathbb{E}$-triangle
$\xymatrix{L_2\ar[r]&W_1\ar[r]&N_1\ar@{-->}[r]&}$ in $\xi$ with $W_1\in\mathcal{W}$ and $L_2\in\mathcal{GW}(\xi)$. So we have the following commutative diagram
$$\xymatrix{
   L_2\ar[d] \ar@{=}[r] & L_2 \ar[d]&\\
  M_2\ar[d] \ar[r] & W_1 \ar[d] \ar[r] & M_1 \ar@{=}[d]\ar@{-->}[r]& \\
  K_2 \ar[r]\ar@{-->}[d] & N_1\ar[r]\ar@{-->}[d] & M_1\ar@{-->}[r]&\\
  &&&  }
$$
where all rows and columns are both $\mathcal{C}(\mathcal{W},-)$-exact and $\mathcal{C}(-,\mathcal{W})$-exact $\mathbb{E}$-triangles in $\xi$ by Proposition \ref{lem8}(1). Proceeding in this manner, we can get both $\mathcal{C}(\mathcal{W},-)$-exact and $\mathcal{C}(-,\mathcal{W})$-exact $\mathbb{E}$-triangles
 $\xymatrix{M_{i+1}\ar[r]&W_i\ar[r]&M_i\ar@{-->}[r]&}$ in $\xi$ with $W_i\in\mathcal{W}$. Spliced these $\mathbb{E}$-triangles together, we obtain the following both $\mathcal{C}(\mathcal{W},-)$-exact and $\mathcal{C}(-,\mathcal{W})$-exact $\xi$-exact complex $$\xymatrix{\cdots\ar[r]&W_1\ar[r]&W_0\ar[r]&X\ar[r]&0}$$
 with $W_i\in\mathcal{W}, i\geq 0$.

 Dually, we can obtain a both
 $\mathcal{C}(\mathcal{W},-)$-exact and $\mathcal{C}(-,\mathcal{W})$-exact $\xi$-exact complex $$\xymatrix{0\ar[r]&X\ar[r]&W^0\ar[r]&W^1\ar[r]&\cdots}$$
 with $W^i\in\mathcal{W}, i\geq 0$. Hence $X\in\mathcal{GW}(\xi)$, as desired.

 (2) Assume that $G\in\mathcal{GW}(\xi)$ and $H$ is a direct summand of $G$, then there exists $H'\in\mathcal{C}$ such that $G=H\oplus H'$. Therefore there exist two split $\mathbb{E}$-triangles  $\xymatrix@C=2em{H\ar[r]^{\tiny\begin{bmatrix}1\\0\end{bmatrix}}&G\ar[r]^{\tiny\begin{bmatrix}0&1\end{bmatrix}\ \ }&H'\ar@{-->}[r]^0&}$ and $\xymatrix@C=2em{H'\ar[r]^{\tiny\begin{bmatrix}0\\1\end{bmatrix}}&G\ar[r]^{\tiny\begin{bmatrix}1&0\end{bmatrix}\ \ }&H\ar@{-->}[r]^0&}$
in $\xi$. Since $G\in\mathcal{GW}(\xi)$, there exists a both $\mathcal{C}(\mathcal{W},-)$-exact and  $\mathcal{C}(-,\mathcal{W})$-exact $\mathbb{E}$-triangle $\xymatrix{G\ar[r]^{\alpha_{-1}}&W_{-1}\ar[r]^{\beta_{-1}}&K_{-1}\ar@{-->}[r]^{\delta_{-1}}&}$ in $\xi$ with $W_{-1}\in\mathcal{W}$ and $K_{-1}\in\mathcal{GW}(\xi)$. It follows from Proposition \ref{lem8}(2) that we have the following commutative diagram

$$\xymatrix@C=3em{
  H\ar[r]^{\tiny\begin{bmatrix}1\\0\end{bmatrix}} \ar@{=}[d] &G \ar[r]^{\tiny\begin{bmatrix}0&1\end{bmatrix}} \ar[d]^{\alpha_{-1}}& H' \ar[d]^{\alpha'_{-1}} \ar@{-->}[r]^0 &  \\
  H\ar[r]^{g_{-1}} & W_{-1} \ar[r]^{f_{-1}} \ar[d]^{\beta_{-1}} & X\ar@{-->}[r]^{\rho_{-1}} \ar[d]^{\beta'_{-1}} &\\
  & K_{-1} \ar@{-->}[d]^{\delta_{-1}} \ar@{=}[r] &K_{-1}\ar@{-->}[d]^{\tiny\begin{bmatrix}0&1\end{bmatrix}_*\delta_{-1}}& \\
  &&&}
$$
  where all rows and columns are both $\mathcal{C}(\mathcal{W},-)$-exact and  $\mathcal{C}(-,\mathcal{W})$-exact  $\mathbb{E}$-triangles in $\xi$.
Note that $\xymatrix@C=2em{H'\ar[r]^{\alpha'_{-1}}&X\ar[r]^{\beta'_{-1}}&K_{-1}\ar@{-->}[r]^{\tiny\begin{bmatrix}0&1\end{bmatrix}_*\delta_{-1}}&}$  and $\xymatrix@C=2em{H'\ar[r]^{\tiny\begin{bmatrix}0\\1\end{bmatrix}}&G\ar[r]^{\tiny\begin{bmatrix}1&0\end{bmatrix}\ \ }&H\ar@{-->}[r]^0&}$ are  $\mathbb{E}$-triangles in $\xi$. Then there exists a commutative diagram
$$\xymatrix@C=3em{
    H' \ar[d]_{\tiny\begin{bmatrix}0\\1\end{bmatrix}} \ar[r]^{\alpha'_{-1}} & X \ar[d]_{g'_{-1}} \ar[r]^{\beta'_{-1}} & K_{-1} \ar@{=}[d] \ar@{-->}[r]^{\tiny\begin{bmatrix}0&1\end{bmatrix}_*\delta_{-1}}&\\
  G \ar[d]_{\tiny\begin{bmatrix}1&0\end{bmatrix}} \ar[r]^{\alpha''_{-1}} & G_{-1} \ar[d]_{f'_{-1}} \ar[r]^{\beta''_{-1}}&K_{-1}\ar@{-->}[r]^{\delta_{-1}'}&\\
    H \ar@{=}[r]\ar@{-->}[d]^0&  H \ar@{-->}[d]^0& & \\
    && &}
$$
where $\xymatrix@C=2em{G\ar[r]^{\alpha''_{-1}}&G_{-1}\ar[r]^{\beta''_{-1}}&K_{-1}\ar@{-->}[r]^{\delta_{-1}'}&}$ and $\xymatrix@C=2em{X\ar[r]^{g'_{-1}}&G_{-1}\ar[r]^{f'_{-1}}&H\ar@{-->}[r]^{0}&}$ are $\mathbb{E}$-triangles in $\xi$  because $\xi$ is closed under cobase change. It is easy to check that the second row is  both $\mathcal{C}(\mathcal{W},-)$-exact and  $\mathcal{C}(-,\mathcal{W})$-exact $\mathbb{E}$-triangle in $\xi$ because the first row, the first and the second columns are both $\mathcal{C}(\mathcal{W},-)$-exact and  $\mathcal{C}(-,\mathcal{W})$-exact. It follows from Lemma \ref{lem10} that   $G_{-1}\in\mathcal{GW}(\xi)$ since $G, K_{-1}\in\mathcal{GW}(\xi)$, hence there exists a both $\mathcal{C}(\mathcal{W},-)$-exact and  $\mathcal{C}(-,\mathcal{W})$-exact $\mathbb{E}$-triangle
$\xymatrix@C=2em{G_{-1}\ar[r]^{\alpha_{-2}}&W_{-2}\ar[r]^{\beta_{-2}}&K_{-2}\ar@{-->}[r]^{\delta_{-2}}&}$ in $\xi$ with $W_{-2}\in\mathcal{W}$ and $K_{-2}\in\mathcal{GW}(\xi)$. Hence there exists a commutative diagram
$$\xymatrix@C=3em{
  X\ar[r]^{g_{-1}'} \ar@{=}[d] &G_{-1} \ar[r]^{f_{-1}'} \ar[d]^{\alpha_{-2}}& H \ar[d]^{\alpha'_{-1}} \ar@{-->}[r]^0 &  \\
  X\ar[r]^{g_{-2}} & W_{-2} \ar[r]^{f_{-2}} \ar[d]^{\beta_{-2}} & Y\ar@{-->}[r]^{\rho_{-2}} \ar[d]^{\beta'_{-2}} &\\
  & K_{-2} \ar@{-->}[d]^{\delta_{-2}} \ar@{=}[r] &K_{-2}\ar@{-->}[d]^{(f'_{-1})_*\delta_{-2}}& \\
  &&&}
$$ where all rows and columns are both $\mathcal{C}(\mathcal{W},-)$-exact and  $\mathcal{C}(-,\mathcal{W})$-exact   $\mathbb{E}$-triangles in $\xi$ by Proposition \ref{lem8}(2). Proceeding this manner,  one can get a both $\mathcal{C}(\mathcal{W},-)$-exact and  $\mathcal{C}(-,\mathcal{W})$-exact  $\xi$-exact complex $\xymatrix@C=1.7em{0\ar[r]&H\ar[r]&W_{-1}\ar[r]& W_{-2}\ar[r]&\cdots}$ with each $W_{-i}\in\mathcal{W}$ for $i\geqslant 1$. Dually, we can get the following both $\mathcal{C}(\mathcal{W},-)$-exact and  $\mathcal{C}(-,\mathcal{W})$-exact $\xi$-exact complex $$\xymatrix@C=2em{\cdots\ar[r]&W_{1}\ar[r]& W_{0}\ar[r]&H\ar[r]&0}$$ with each $W_{i}\in\mathcal{W}$ for $i\geqslant 0$. Hence $H\in\mathcal{GP}(\xi)$, as desired.
 \end{proof}

By Remark \ref{rem:4.2}(1) and Theorem \ref{thm:stability}, we have the following corollary.

\begin{cor}\label{corllary2} \emph{(}see \cite[Theorems 4.1 and 4.6(2)]{Hua}\emph{)} Assume that $\mathcal{C}$ is an abelian category and $\xi$ is the class of exact sequences. If $\mathcal{W}$ is a full additive subcategory of $\mathcal{C}$ that is closed under isomorphisms, then $\mathcal{GW}^n(\xi)=\mathcal{GW}(\xi)$ for any $n\geqslant1$ and $\mathcal{GW}(\xi)$ is closed under direct summands.
\end{cor}

As a consequence of Remark \ref{rem:4.2}(2) and Theorem \ref{thm:stability}, we have the following corollary.
\begin{cor}\label{corllary3} Assume that $\mathcal{C}$ is a triangulated category and $\xi$ is a proper class of triangles. If $\mathcal{W}$ is an additive full subcategory of $\mathcal{T}$ closed under isomorphisms and $\Sigma$-stable, then $\mathcal{GW}^n(\xi)=\mathcal{GW}(\xi)$ for any $n\geqslant1$ and $\mathcal{GW}(\xi)$ is closed under direct summands.
\end{cor}
\begin{rem} \label{rem:4.19} We note that Corollary \ref{corllary3} refines one of main results of Yang and Wang in \cite{Yang3}. In their paper, they showed that for any triangulated category with countable coproducts, if the class $\mathcal{W}$ is closed under countable coproducts, then $\mathcal{GW}^n(\xi)=\mathcal{GW}(\xi)$ for any $n\geqslant1$ and $\mathcal{GW}(\xi)$ is closed under direct summands, see \cite[Theorem 3.3]{Yang3}.
\end{rem}

%
%

\begin{cor}\label{pro2} Given both $\mathcal{C}(\mathcal{W},-)$-exact and $\mathcal{C}(-,\mathcal{W})$-exact $\mathbb{E}$-triangle
$$\xymatrix{X\ar[r]&Y\ar[r]&Z\ar@{-->}[r]&}$$ in $\xi$ with $Y\in\mathcal{GW}(\xi)$, then $X\in\mathcal{GW}(\xi)$ if and only if $Z\in\mathcal{GW}(\xi)$.
\end{cor}
\begin{proof}
Assume that $Y,Z\in\mathcal{GW}(\xi)$, then Theorem \ref{thm4} implies that $X$ has a coproper $\mathcal{W}(\xi)$-resolution which is $\mathcal{C}(\mathcal{W},-)$-exact. Consider a both $\mathcal{C}(\mathcal{W},-)$-exact and $\mathcal{C}(-,\mathcal{W})$-exact $\mathbb{E}$-triangle
$\xymatrix{L\ar[r]&W\ar[r]&Y\ar@{-->}[r]&}$ in $\xi$ with $W\in\mathcal{W}$ and  $L\in\mathcal{GW}(\xi)$, then we have the following commutative diagram
$$\xymatrix{
   L\ar[d] \ar@{=}[r] & L \ar[d]&\\
  Z'\ar[d] \ar[r] & W \ar[d] \ar[r] & Z \ar@{=}[d]\ar@{-->}[r]& \\
  X \ar[r]\ar@{-->}[d] & Y \ar[r]\ar@{-->}[d] & Z\ar@{-->}[r]&\\
  &&&  }
$$
where all rows and columns are both $\mathcal{C}(\mathcal{W},-)$-exact and $\mathcal{C}(-,\mathcal{W})$-exact $\mathbb{E}$-triangles in $\xi$ by Proposition \ref{lem8}(1) . Note that $Z\in\mathcal{GW}(\xi)$, there exists a both $\mathcal{C}(\mathcal{W},-)$-exact and $\mathcal{C}(-,\mathcal{W})$-exact $\mathbb{E}$-triangle  $\xymatrix{K\ar[r]&V\ar[r]&Z\ar@{-->}[r]&}$ in $\xi$ with $V\in\mathcal{W}$ and $K\in\mathcal{GW}(\xi)$. So we have the following commutative diagram
$$\xymatrix{
   &K \ar[d] \ar@{=}[r] & K \ar[d]\\
  Z'\ar@{=}[d] \ar[r] & N\ar[d] \ar[r] & V \ar[d]\ar@{-->}[r]& \\
  Z'\ar[r] & W \ar[r]\ar@{-->}[d] & Z\ar@{-->}[r]\ar@{-->}[d]&\\
  &&&  }
$$
where all rows and columns are $\mathbb{E}$-triangles in $\xi$.   Moveover the second row and column are $\mathcal{C}(\mathcal{W},-)$-exact by Lemma \ref{lem3}(1). Hence $K\oplus W\cong Z'\oplus V$, which implies that $Z'\in \mathcal{GW}(\xi)$ because $\mathcal{GW}(\xi)$ is closed under direct summands by Theorem \ref{thm:stability}(2). Hence $X$ has a proper $\mathcal{W}(\xi)$-resolution which is $\mathcal{C}(-, \mathcal{W})$-exact by Theorem \ref{thm3}. Therefore $X\in\mathcal{GW}(\xi)$.

Dually, we can prove that $Z\in\mathcal{GW}(\xi)$ whenever $X, Y\in\mathcal{GW}(\xi)$.
\end{proof}
\vspace{1cm}

\renewcommand\refname

\end{document}